\documentclass{article}
\usepackage[a4paper, total={6in, 8.6in}]{geometry}
\usepackage{amsmath,amsthm,amssymb,mathrsfs,amstext, titlesec,enumitem, comment, graphicx, color, xcolor, stmaryrd,mathabx, lipsum,filecontents}
\usepackage{hyperref}
\usepackage{xpatch}
\hypersetup{colorlinks,linkcolor={magenta},citecolor={blue}}
\usepackage{xpatch}
\usepackage{tikz-cd}
\makeatletter
\def\Ddots{\mathinner{\mkern1mu\raise\p@
\vbox{\kern7\p@\hbox{.}}\mkern2mu
\raise4\p@\hbox{.}\mkern2mu\raise7\p@\hbox{.}\mkern1mu}}
\makeatother

\titleformat{\section}
{\normalfont\Large\bfseries}{\thesection}{1em}{}
\titleformat*{\subsection}{\Large\bfseries}
\titleformat*{\subsubsection}{\large\bfseries}
\titleformat*{\paragraph}{\large\bfseries}
\titleformat*{\subparagraph}{\large\bfseries}

\makeatother
\theoremstyle{Theorem}
\newtheorem{thm}{Theorem}[section]
\newtheorem{lem}[thm]{Lemma}
\newtheorem{problem}[thm]{Problem}

\newtheorem{cor}[thm]{Corollary}
\newtheorem{conjecture}[thm]{Conjecture}

\theoremstyle{definition}
\newtheorem{defn}[thm]{Definition}

\newcommand{\N}{\mathbb{N}}
\newcommand{\Z}{\mathbb{Z}}

\date{\vspace{-5ex}}

\begin{document}

\title{{\bf On Difference Sets of Dense Subsets of $\mathbb{Z}^2$}}
\author{Sayan Goswami\\  \textit{sayan20math@gmail.com}\footnote{Ramakrishna Mission Vivekananda Educational and Research Institute, Belur, Howrah, 711202, India}}

\makeatother

\maketitle
\begin{abstract}
In this article, we study the structure of the difference set $E - E$ for subsets $E \subseteq \mathbb{Z}^2$ of positive upper Banach density. Fish asked in \cite[Problem~2]{fish} whether, for every such set $E$, there exists a nonzero integer $k$ such that
\[
k \cdot \mathbb{Z} \subseteq \{\, xy : (x,y) \in E - E \,\}.
\]
Although this question remains open, we establish a relatively weaker form of this conjecture. Specifically, we prove that if $\langle a_j\rangle_{j=1}^m$ is any finite sequence in $\mathbb{N},$ then there exist infinitely many integers $k \in \mathbb{Z}$ and a sequence $\langle x_n \rangle_{n \in \mathbb{N}}$ in $\mathbb{Z}$ such that
\[
k \cdot MT\left(\langle a_j \rangle_{j=1}^m, \langle x_n\rangle_{n}\right) \subseteq \{\, xy : (x,y) \in E - E \,\},
\]
where $MT\left(\langle a_j \rangle_{j=1}^m, \langle x_n\rangle_{n}\right)$ denotes the milliken-Taylor configuration generated by the sequences $\langle a_j\rangle_{j=1}^m$ and $\langle x_n \rangle_{n \in \mathbb{N}}$.
\end{abstract}


\section{Introduction}
Estimating the sizes of sumsets and product sets of subsets of $\mathbb{Z}$ is a central topic in additive combinatorics. For any two sets $A,B \subset \mathbb{Z}$ (or $\mathbb{R}$), we define their sumset and product set by
\[
A+B=\{a+b : a\in A,\ b\in B\}
\quad \text{and} \quad
A\cdot B=\{ab : a\in A,\ b\in B\}.
\]
One of the most fundamental open problems in sum--product theory is the following conjecture of Erd\H{o}s and Szemer\'{e}di \cite{erdos}.

\begin{conjecture}[Erd\H{o}s--Szemer\'{e}di \cite{erdos}]
If $A \subset \mathbb{Z}$ (or $\mathbb{R}$) is a finite set, then for every $\epsilon>0$,
\[
|A+A| + |A\cdot A| \gg |A|^{2-\epsilon}.
\]
\end{conjecture}

However, if one replaces the finite set $A$ by an infinite one, then an immediate question appears what kind of structures are preserved under the combination of these sum-product of sets.
For various questions on sum, difference, and product sets, we refer to the article \cite{ruzsa}.
In \cite{fishold, fish},  Bj\H{o}rklund and Fish initiated an infinitary version of these types of questions. For some recent developments, we refer to the article \cite{BF}.

\begin{defn}[Upper Banach density in $\mathbb{Z}^2$]
Let $A \subseteq \mathbb{Z}^2$. The \emph{upper Banach density} of $A$ is defined by
\[
d^\star(A)
:=
\lim_{N \to \infty}
\sup_{(m,n)\in \mathbb{Z}^2}
\frac{
\left|
A \cap \bigl( (m,n) + [-N,N]^2 \bigr)
\right|
}{
(2N+1)^2
}.
\]
\end{defn}

In \cite{fish}, Fish asked the following question.

\begin{problem}\textup{\cite[Problem 2 (page 5)]{fish}}\label{qn}
Is it true that for every set \(E \subset \mathbb{Z}^2\) of positive upper Banach density,
there exists a constant \(k_0\), depending only on \(d^*(E)\), such that for some integer
\(k \le k_0\),
\[
k\mathbb{Z} \subseteq \{\, xy : (x,y) \in E - E \,\}?
\]
\end{problem}

Although this problem is still open, we prove a relatively weaker version of this result. But before we state our result, let us recall the notions of $IP$ sets. Let $\mathcal{P}_{f}\left(\mathbb{N}\right)$ be the collection of nonempty finite subsets of $\mathbb{N}.$ In any commutative semigroup $(S,+),$ a set $A$ is said to be an $IP$ set if there exists a sequence $\langle x_n\rangle_n$ such that $$A=FS\left(\langle x_{n}\rangle_{n\in\mathbb{N}}\right)=\left\{ \sum_{t\in H}x_{t}:H\in\mathcal{P}_{f}\left(\mathbb{N}\right)\right\}.$$ Over the semigroup $(\N,\cdot ),$ we denote it by $FP\left(\langle x_{n}\rangle_{n\in\mathbb{N}}\right).$
The notions of $IP$ sets have significant importance in Ramsey theory, for ex. it is known that they are partition regular (see \cite{sum}).

An extension of Hindman's Finite Sum Theorem was studied independently by Milliken and Taylor in \cite{m,t}. To address their results, we need the following notions of {\it Milliken-Taylor system.} 

\begin{defn}[Milliken-Taylor system]\label{defmillikentaylor}
    Let $m\in \N,$ let $\langle a_j\rangle_{j=1}^m$ and $\langle x_n\rangle_{n}$ be sequences in $\N.$ The
Milliken-Taylor System determined by $\langle a_j\rangle_{j=1}^m$ and $\langle x_n\rangle_{n}$ is 
$$MT\left(\langle a_j \rangle_{j=1}^m, \langle x_n\rangle_{n}\right)=\left\lbrace \sum_{j=1}^m a_j \cdot \sum_{t\in F_j}x_t:F_1<F_2<\cdots <F_m , \forall F_i\in \mathcal{P}_{f}\left(\mathbb{N}\right)\right\rbrace.$$
\end{defn}

So clearly, for any sequence $\langle x_n\rangle_{n},$ $FS(\langle x_n\rangle_{n})=MT\left(\{1\}, \langle x_n\rangle_{n}\right).$ However, the following theorem says a more general result than the Finite Sum Theorem. Although in the original text the result was described over the set of positive integers, the technique for $\Z$ is similar.

\begin{thm}\label{mt}\cite[Theorem 17.31]{book}
    Let $\langle a_j\rangle_{j=1}^m$ be sequences in $\N.$ Then for any $p\in E(\beta \Z,+)$ and any $A\in a_1p+\cdots +a_mp,$ there exists a sequence $\langle x_n\rangle_{n}$ such that $MT\left(\langle a_j \rangle_{j=1}^m, \langle x_n\rangle_{n}\right)\subseteq A.$
\end{thm}

In this article, we prove the following theorem, which is a weaker version of the Problem \ref{qn}.

\begin{thm}\label{new1}
  Let $m\in \N,$ and $\langle a_j\rangle_{j=1}^m$ be a sequence in $\N$ is  given. If  \(E \subset \mathbb{Z}^2\) is a set of positive upper Banach density, then there exist infinitely many $k\in \Z$ and infinitely many sequences $\langle x_n\rangle_n$ in $\Z$ such that 
  \[
k\cdot MT\left(\langle a_j \rangle_{j=1}^m, \langle x_n\rangle_{n}\right) \subseteq \{\, xy : (x,y) \in E - E \,\}.
\]
\end{thm}

For any sequence $\langle x_n\rangle_n$ in $\Z$ and $H\in \mathcal{P}_{f}\left(\mathbb{N}\right),$ define $x_H=\sum_{t\in H}x_t.$ Our technique also proves the following result. 

\begin{thm}\label{new2}
  If  \(E \subset \mathbb{Z}^2\) is a set of positive upper Banach density, then there exist infinitely many $k\in \Z$ and infinitely many sequences $\langle x_n\rangle_n$ in $\Z$ such that 
  \[
\{x_K-x_H: H,K\in \mathcal{P}_{f}\left(\mathbb{N}\right),\text{ and }\max H<\min K\} \subseteq \{\, x+y : (x,y) \in E - E \,\}.
\]
\end{thm}

\section{Preliminaries}

\subsection{Ergodic preliminaries}

In this article, we intend to study these notions of recurrence for countable
amenable groups. The reason for restricting to countable amenable groups is that
they provide a natural framework for Furstenberg's correspondence principle. 

\begin{defn}
A discrete group $G$ is said to be \emph{amenable} if there exists an invariant
mean on the space $\mathcal{B}(G)$ of real-valued bounded functions on $G$, that
is, a positive linear functional $L : \mathcal{B}(G) \to \mathbb{R}$ satisfying:
\begin{enumerate}
\item $L(\mathbf{1}_G) = 1$;
\item $L(f_g) = L({}_g f) = L(f)$ for all $f \in \mathcal{B}(G)$ and all
$g \in G$, where $f_g(t) = f(tg)$ and ${}_g f(t) = f(gt)$.
\end{enumerate}
\end{defn}

The existence of an invariant mean is only one of many equivalent
characterizations of amenability. In what follows, we shall make frequent use
of the Følner characterization of amenability for discrete groups, which will
be particularly useful for our purposes.

\begin{thm}
A countable group $G$ is amenable if and only if it admits a left Følner
sequence, that is, a sequence of finite nonempty sets
$(F_n)_{n \in \mathbb{N}} \subseteq G$ such that for every $g \in G$,
\[
\lim_{n \to \infty} \frac{|gF_n \triangle F_n|}{|F_n|} = 0.
\]
\end{thm}

\begin{defn}
Let $G$ be a countable amenable group. A subset $E \subseteq G$ is said to have

\begin{enumerate}
\item \emph{positive Følner density}, or \emph{positive upper density with respect to
a left Følner sequence} $(F_n)_{n \in \mathbb{N}}$, if
\[
\limsup_{n \to \infty} \frac{|E \cap F_n|}{|F_n|} > 0.
\]
This quantity will be denoted by $d_{F_n}(E)$.

\item The \emph{upper Banach density} of $E$ is defined as $d^\star (E)=\sup_{(F_n)}d_{F_n}(E),$
where the supremum is taken over all Følner sequences $(F_n)$ in $G$.
\end{enumerate}

\end{defn}

Note that Central sets can be defined as members of minimal idempotent ultrafilters, and Central* sets are those sets that sit in every minimal idempotent ultrafilter. We will recall these notions in the coming section
The following theorem, which follows from~\cite[Theorem 1.4]{BM}, shows that the set of return times
is large. 

\begin{thm}\label{measure}
Let $G$ be a countable amenable group, let $\{T_g\}_{g \in G}$ and
$\{S_g\}_{g \in G}$ be commuting measure preserving actions of $G$ on a
probability space $(X,\mathcal{B},\mu)$, and let $A \in \mathcal{B}$ satisfy
$\mu(A) > 0$. Then the set
\[
\{ g \in G : \mu(A \cap T_g^{-1}A \cap (T_g S_g)^{-1}A) > 0 \}
\]
is Central* in $G$.
\end{thm}

From the above theorem it is also clear that $\{ g \in G : \mu(A \cap S_g^{-1}A \cap (T_g S_g)^{-1}A) > 0 \}$ is Central* in $G$. Now, from the Furstenberg correspondence principle over amenable groups, and Theorem \ref{measure}, it follows the following corollary.

\begin{cor}\label{corfirst}
   Let $G$ be a countable left amenable group, and $E\subseteq G\times G$ be such that $d^\star (E)>0.$ Then 
\[
\{x\in G: d\bigl(E \cap (x,e)^{-1}E \cap (x,x)^{-1}E\bigr) > 0\}
\]
is a central* set. And similarly,
\[
\{x\in G: d\bigl(E \cap (e,x)^{-1}E \cap (x,x)^{-1}E\bigr) > 0\}
\]
is a central* set.
\end{cor}

Here, we define the notions of density recurrent sets, which will be useful in our study.

\begin{defn}\label{def}
Let $G$ be a countable left amenable group.
\begin{enumerate}

\item
Let $B \subseteq G$. Then $B$ is \emph{left density $2$-recurrent} if and only if
whenever $E \subseteq G \times G$ satisfies $d(E) > 0$, there exists $x \in B$
such that
\[
d\bigl(E \cap (x,e)^{-1}E \cap (x,x)^{-1}E\bigr) > 0.
\]

\item
Let $C \subseteq G$. Then $B$ is \emph{right density $2$-recurrent} if and only if
whenever $E \subseteq G \times G$ satisfies $d(E) > 0$, there exists $x \in B$
such that
\[
d\bigl(E \cap (e,x)^{-1}E \cap (x,x)^{-1}E\bigr) > 0.
\]

\item
Let $D \subseteq G$. Then $B$ is \emph{both sided density $2$-recurrent} if and only if
whenever $E \subseteq G \times G$ satisfies $d(E) > 0$, there exists $x \in B$
such that
\begin{enumerate}
    \item $d\bigl(E \cap (e,x)^{-1}E \cap (x,x)^{-1}E\bigr) > 0; \text{ and } $
    \item $d\bigl(E \cap (x,e)^{-1}E \cap (x,x)^{-1}E\bigr) > 0.$
\end{enumerate}
\end{enumerate}
\end{defn}

\subsection{A brief review of Topological Algebra}
Here we recall some basic facts about ultrafilters. For details, refer to the book \cite{book}.

Let  $(S,\cdot)$ be any discrete semigroup. Let $\beta S$ be the collection of all ultrafilters. For every $A\subseteq S,$ define $\overline{A}=\{p\in \beta S: A\in p\}.$ Now one can check that the collection $\{\overline{A}: A\subseteq S\}$ forms a basis for a topology. This basis generates a topology over $\beta S.$  We can extend the operation $``\cdot "$ of $S$ over $\beta S$  as: for any $p,q\in \beta S,$ $A\in p\cdot q$ if and only if $\{x:x^{-1}A\in q\}\in p,$ where $x^{-1}A=\{y\in S: xy\in A\}.$ With this operation $``\cdot "$, $(\beta S,\cdot)$ becomes a compact Hausdorff right topological semigroup. One can show that $\beta S$ is nothing but the Stone-\v{C}ech compactification of $S.$ Hence Ellis's theorem guarantees that there exist idempotents in $(\beta S,\cdot)$. The set of all idempotents in $(\beta S,\cdot)$ is denoted by $E\left((\beta S,\cdot)\right).$
It can be shown that every member of the idempotents of $(\beta S,\cdot)$ contains an $IP$ set, which means every idempotent witnesses  Hindman's theorem. Using Zorn's lemma, one can show that $(\beta S,\cdot)$ contains minimal left ideals (minimal w.r.t. the inclusion).  A well-known fact is that the union of such minimal left ideals is a minimal two-sided ideal, denoted by $K(\beta S,\cdot).$ If our $S$ is a group, then for any ultrafilter $p$ define $p^{-1}=\{A^{-1}:A\in p\}.$ It can be verified that $p^{-1}$ is also an ultrafilter. Here we recall a few well-known classes of sets that are relevant to our work. 

\begin{defn}\label{rev2defn}\text{}
\begin{enumerate}
 \item $A$ is {\it central set} if it belongs to a minimal idempotent in $\beta S$.
\item $A$ is {\it central*} if it belongs to all minimal idempotents in $\beta S$.
\end{enumerate}
\end{defn}
The tensor product of two ultrafilters is defined as follows.

\begin{defn} Let $(S,\cdot)$ and $(T,\cdot)$ be two discrete semigroups.
    Let $p\in (\beta S,\cdot )$ and  $q\in (\beta T,\cdot )$ be two ultrafilters. Then the tensor product of $p$ and $q$ is defined as
$$p\otimes q=\{A\subseteq S\times T: \{x\in S: \{y:(x,y)\in A\}\in q\}\in p\}$$
where $x\in S$, $y\in T.$
\end{defn}

\section{Our Results}

For any countable amenable group $(G,\cdot),$ define the following notions.
\begin{defn} Define
\begin{enumerate}
\item $\mathcal{LDR}_2(G)=\{ p \in \beta G : (\forall B \in p)\; B \text{ is left density $2$-recurrent} \};$
\item $\mathcal{RDR}_2(G)=\{ p \in \beta G : (\forall B \in p)\; B \text{ is right density $2$-recurrent} \};$
\item $\mathcal{DR}_2(G)=\{ p \in \beta G : (\forall B \in p)\; B \text{ is both density $2$-recurrent} \}.$
\end{enumerate}
\end{defn}

To obtain Theorem \ref{new1} and Theorem \ref{new2}, we will study the properties of the set $AA^{-1}$, where $A\subseteq G\times G$ with positive upper Banach density. First, we prove a sequence of lemmas. 

\begin{lem}\label{subsemigroup}
Let $G$ be a countable amenable group. Then all $ \mathcal{LDR}_2(G)$ and $\mathcal{RDR}_2(G)$ are subsemigroups of $\beta G$ containing the minimal idempotents of $\beta G$.
\end{lem}

\begin{proof} Here we prove the result $\mathcal{LDR}_2(G).$ However the proof for  $\mathcal{RDR}_2(G)$ is similar and we leave it.
By~\ref{corfirst}, whenever
$E \subseteq G \times G$ satisfies $d^\star(E) > 0$, the set
\[
\{ x \in G : d^\star(E \cap (x,e)^{-1}E \cap (x,x)^{-1}E) > 0 \}
\]
is central*.

Let $p$ be a minimal idempotent in $\beta G$ and let $C \in p$. Then for every
$E \subseteq G \times G$ with $d^\star(E) > 0$,
\[
C \cap \{ x \in G : d^\star(E \cap (x,e)^{-1}E \cap (x,x)^{-1}E) > 0 \} \neq \varnothing.
\]
Hence, there exists $g \in C$ such that
\[
d^\star(E \cap (g,e)^{-1}E \cap (g,g)^{-1}E) > 0.
\]
Thus $C$ is density $2$-recurrent, which implies
$K(\beta G) \subseteq \mathcal{DR}_2(G)$. In particular,
$\mathcal{DR}_2(G)$ is nonempty.

Now let $p,q \in \mathcal{DR}_2(G)$. Let $B \in pq$ and
$A \subseteq G \times G$ with $d^\star(A) > 0$. Define
\[
C = \{ x \in G : x^{-1}B \in q \}.
\]
Then $C \in p$, and hence $C$ is density $2$-recurrent. Choose $x \in C$ such that
\[
d^\star(A \cap (x,e)^{-1}A \cap (x,x)^{-1}A) > 0,
\]
and set
\[
D = A \cap (x,e)^{-1}A \cap (x,x)^{-1}A.
\]
Since $x^{-1}B \in q$, there exists $y \in x^{-1}B$ such that
\[
d^\star(D \cap (y,e)^{-1}D \cap (y,y)^{-1}D) > 0.
\]
As $xy \in B$, it suffices to show that
\[
D \cap (y,e)^{-1}D \cap (y,y)^{-1}D
\subseteq
A \cap (xy,e)^{-1}A \cap (xy,xy)^{-1}A.
\]

Let $(u,v)$ belong to the left-hand side. Then $(u,v) \in D \subseteq A$ and
\[
(u,v) \in (y,e)^{-1}D \subseteq (y,e)^{-1}(x,e)^{-1}A,
\]
so $(xyu,v) \in A$, which implies $(u,v) \in (xy,e)^{-1}A$.
Similarly,
\[
(u,v) \in (y,y)^{-1}D \subseteq (y,y)^{-1}(x,x)^{-1}A,
\]
so $(xyu,xyv) \in A$, and hence $(u,v) \in (xy,xy)^{-1}A$.
Therefore,
\[
(u,v) \in A \cap (xy,e)^{-1}A \cap (xy,xy)^{-1}A.
\]
\end{proof}

Note that from corollary \ref{corfirst}, it is clear that $E(K(\beta G,\cdot ))\subseteq \mathcal{DR}_2(G).$
Although we don't know for $p\in \mathcal{LDR}_2(G)$ if $p^{-1}\in \mathcal{LDR}_2(G),$ or the same for $ \mathcal{RDR}_2(G)$. But here we show this result holds for $\mathcal{DR}_2(G).$

\begin{lem}\label{inverse}
    If $p\in \mathcal{DR}_2(G),$ then $p^{-1}\in \mathcal{DR}_2(G).$ 
\end{lem}
\begin{proof}
    Let $p\in \mathcal{DR}_2(G),$ and $B\in p.$ It will be sufficient to show that $B^{-1}$ is a both density $2$-recurrent set. Pick any set $A$ with $d(A)>0.$ As $p\in \mathcal{DR}_2(G),$  there exists $x\in B$ such that $d(A \cap (x,e)^{-1}A \cap (x,x)^{-1}A)>0$ and  $d(A \cap (e,x)^{-1}A \cap (x,x)^{-1}A)>0.$ Define 

\begin{enumerate}
    \item $D_1=A \cap (x,e)^{-1}A \cap (x,x)^{-1}A;$
    \item $D_2=A \cap (e,x)^{-1}A \cap (x,x)^{-1}A.$
\end{enumerate}
Let 
\begin{enumerate}
    \item $E_1=(x,x)\cdot D_1=A\cap (e,x)\cdot A\cap (x,x)\cdot A;$ and
    \item $E_2=(x,x)\cdot D_2=A\cap (x,e)\cdot A\cap (x,x)\cdot A.$
\end{enumerate}
Then both $E_1$ and $E_2$ have positive density. Hence $d(A\cap (e,x^{-1})^{-1} A\cap (x^{-1},x^{-1})^{-1} A)>0$ and $d(A\cap (x^{-1},e)^{-1} A\cap (x^{-1},x^{-1})^{-1} A)>0.$ This completes the proof.
\end{proof}

As both $\mathcal{LDR}_2(G)$ and $ \mathcal{RDR}_2(G)$ contains $ \mathcal{DR}_2(G),$
from lemma \ref{subsemigroup} and lemma \ref{inverse} it is clear that if $p\in \mathcal{DR}_2(G)$ and $q\in \mathcal{LDR}_2(G),$ then $p^{-1}q\in \mathcal{LDR}_2(G).$ The same conclusion holds for $ \mathcal{RDR}_2(G).$ 

Let $$T=\bigcap \left\lbrace \overline{EE^{-1}}: A\subseteq G\times G\, \text{ and } \, d^\star(E)>0 \right\rbrace.$$

The following lemma shows that the tensor product of two elements from $\mathcal{DR}_2(G)$ is a member of $T.$

\begin{lem}\label{tensor}
    If $p,q\in \mathcal{DR}_2(G),$ then $ p\otimes q\in T.$ 
\end{lem}
\begin{proof}
Here, we only prove $p\otimes q\in T$ rest proofs are similar.
    If possible, assume that $p\otimes q\notin T.$ Hence there exists a set $A$ with $d(A)>0$ such that $A^{-1}A\notin p\otimes q.$ Hence $B=(A^{-1}A)^c\in p\otimes q.$ But $B\in p\otimes q$ implies $C=\{x:\{y:(x,y)\in B\}\in q\}\in p.$ Choose $x\in C$ such that $d(A\cap (x,e)^{-1}A\cap (x,x)^{-1}A)>0.$ Let $$D=A\cap (x,e)^{-1}A\cap (x,x)^{-1}A.$$ Now there exists $y\in B_x=\{(y:(x,y)\in B\}$ such that $d(D\cap (e,y)^{-1}D\cap (y,y)^{-1}D)>0.$ But this implies $d(A\cap (x,y)^{-1}A)>0.$ But this implies $(x,y)\in A^{-1}A\cap B,$ a contradiction.
\end{proof}

\begin{lem}\label{otimes} The above proof suggests something more: if $p\in \mathcal{LDR}_2(G)$ and $ q\in \mathcal{RDR}_2(G),$ then $p\otimes q\in T.$ 

That means if $E\subseteq \Z\times \Z$ is a subset with positive upper Banach density, then for any $p\in \mathcal{LDR}_2(G)$ and $ q\in \mathcal{RDR}_2(G),$  we have $(E-E)\in p\otimes q.$
\end{lem}

To prove Theorem \ref{new1}, we need the following lemma. Again, the reference contains the proof over the set of positive integers, but the same proof holds over the set of all integers.

\begin{lem}\label{product}\cite[Lemma 5.19.2]{book}
    If $p\in E(K(\beta \Z,+))$ and $n\in \N,$ then $n\cdot p\in E(K(\beta \Z,+)).$
\end{lem}

Now we prove our first Theorem.

\begin{proof}[Proof of Theorem \ref{new1}]
Define the map $f:\Z^2\to \Z$ by $f(x,y)=xy.$ Let $\beta f:\beta \Z^2\to \beta \Z$ be its continuous extension. First, we claim that $\beta f(p\otimes q)=p\cdot q.$ As both are ultrafilters, it will be sufficient to show that $\beta f(p\otimes q) \subseteq p\cdot q.$ Suppose $B\in \beta f(p\otimes q).$ Then $f^{-1}(B)\in p\otimes q. $ Then $\{a:\{b: (a,b)\in f^{-1}(B)\}\in q\}\in p.$ But $(a,b)\in f^{-1}(B)\implies a\cdot b\in B.$ But this implies $B\in p\cdot q.$ And this proves the claim.

Now choose any $p,q\in E(K(\beta \Z,+))$. Then from Lemmas \ref{product}, \ref{subsemigroup} we have $a_1q+\cdots +a_mq\in \mathcal{RDR}_2(\Z).$ But from Lemma \ref{otimes} and our claim, it follows that $$f(E-E)\in \beta f(p\otimes (a_1q+\cdots +a_mq))=p\cdot (a_1q+\cdots +a_mq).$$ Hence $$\{\, xy : (x,y) \in E - E \,\}\in p\cdot (a_1q+\cdots +a_mq).$$ Now from Theorem \ref{mt}, it is an routine exercise to check that  sequences $\langle x_n\rangle_n$ in $\Z$ such that 
  \[
k\cdot MT\left(\langle a_j \rangle_{j=1}^m, \langle x_n\rangle_{n}\right) \subseteq \{\, xy : (x,y) \in E - E \,\}.
\]
    
\end{proof}

\begin{proof}[Proof of Theorem \ref{new2}]
 Suppose $E\subseteq \Z^2$ is the set with positive upper Banach density. Define the map $f:\Z^2\to \Z$ by $f(x,y)=x+y.$ Let $\beta f:\beta \Z^2\to \beta \Z$ be its continuous extension. Likewise, the above proof it is immediate that for any $p,q\in \beta \Z,$ $\beta f(p\otimes q)=p+q.$ Now choose any $p\in  E(K(\beta \Z,+)).$ Then from Lemmas \ref{inverse}, \ref{otimes} it follows that $(E-E)\in (-p)\otimes p.$ And so, 
 \[
 f(E-E)=\{\, x+y : (x,y) \in E - E \,\}\in (-p)+p.
 \]

So to prove our result, it will be sufficient to show that all the members of $(-p)+p$ contain the desired pattern.
Let $A\in (-p)+p.$ Then $\{x:-x+A\in p\}\in (-p).$ Hence $B=\{y: y+A\in p\}\in p.$ Note $B^\star=\{y:-y+B\in p\}\in p.$ Pick $y_1\in B^\star.$ From \cite[Lemma 4.14]{book}, $-y_1+B^\star \in p.$
Then pick $$y_2\in B_2=B^\star \cap (-y_1+B^\star)\cap (y_1+A)\in p.$$ Hence $y_2-y_1\in A.$ Now pick $$y_3\in B^\star \cap \bigcap_{y\in FS({y_1,y_2})}(-y+B^\star)\cap \bigcap_{y\in FS({y_1,y_2})}(y+A)\cap (-y_2+B_2^\star)\cap (-y_1+B_1^\star)\in p.$$

Then clearly for any $y\in FS({y_1,y_2})$, $y_3-y\in A.$ Also from the choice $y_3\in -y_2+B_2^\star,$ we have $y_2+y_3-y_1\in A.$ And continuing this iteration method, one can prove our result.

\end{proof}

 \section*{Acknowledgment} The author of this paper is supported by NBHM postdoctoral fellowship with reference no: 0204/27/(27)/2023/R \& D-II/11927.


\begin{thebibliography}{10}

\bibitem{BM} V.~Bergelson and R.~McCutcheon, \emph{Central sets and a non-commutative Roth theorem}, Amer. J. Math. \textbf{129} (2007), no.~5, 1251--1275 \href{https://doi.org/10.1353/AJM.2007.0031}{link}.

\bibitem{BF} K.~Bulinski and A.~Fish, \emph{Quantitative twisted patterns in positive density subsets}, Discrete Analysis, 2024:1, 17 pp. \href{https://doi.org/10.48550/arXiv.2102.05862}{link}.

\bibitem{fishold}  M. Bj\H{o}rklund and A. Fish, \emph{Characteristic polynomial
patterns in difference sets of matrices}, Bull. London Math. Soc. (2016)
48 (2): 300-308 \href{https://doi.org/10.1112/blms/bdw008}{link}.


\bibitem{erdos} P. Erd\H{o}s and E. Szemer\'{e}di, \emph{On sums and
products of integers}, Studies in pure mathe- matics, 213-218, Birkh\"{a}user,
Basel, 1983 \href{https://doi.org/10.1007/978-3-0348-5438-2_19}{link}.

\bibitem{fish} A. Fish, \emph{On product of difference sets of positive
density}, Proc. Amer. Math. Soc. 146 (2018), 3449-3453 \href{http://dx.doi.org/10.1090/proc/14078}{link}.

\bibitem{sum} N. Hindman: Finite sums from sequences within cells
of partitions of $\mathbb{N}$, J. Combin. Theory Ser. A, $17\, (1974),\,
1-11$ \href{https://doi.org/10.1016/0097-3165(74)90023-5}{link}.

\bibitem{book} N. Hindman and D. Strauss, \textit{Algebra in the Stone–Čech Compactification: Theory and Applications}, 2nd ed., Walter de Gruyter \& Co., Berlin, 2012.

\bibitem{m} K. Milliken: Ramsey’s Theorem with sums or unions, J. Comb. Theory (Series A) $18$ $(1975)$,
$276$-$290$ \href{https://doi.org/10.1016/0097-3165(75)90039-4}{link}.

\bibitem{ruzsa} I. Z. Ruzsa, \emph{Sumsets and structure}, Lecture notes, \href{https://www.math.cmu.edu/~af1p/Teaching/AdditiveCombinatorics/Additive-Combinatorics.pdf}{link}

\bibitem{t}  A. Taylor: A canonical partition relation for finite subsets of $\omega$, J. Comb. Theory (Series A) $21 (1976), 137$-$146$ \href{https://doi.org/10.1016/0097-3165(76)90058-3}{link}.

\end{thebibliography}
\end{document}